\documentclass[10pt]{article}

\usepackage{amsfonts}
\usepackage{amssymb}
\usepackage{amsthm}
\usepackage{amsmath}
\usepackage{bm}
\usepackage{amscd} 
\usepackage{graphicx}
\usepackage{mathrsfs}
\usepackage{relsize}
\usepackage{amsrefs}
\usepackage{comment}
\usepackage[all]{xy}
\usepackage{MnSymbol}

\usepackage{tikz-cd}

\usepackage{fullpage}
\usepackage{hyperref}

\numberwithin{equation}{section} \DeclareMathSizes{2}{10}{12}{13}
\newtheorem{theorem}{Theorem}[section]
\newtheorem{lemma}[theorem]{Lemma}

\newtheorem{prop}[theorem]{Proposition}

\newtheorem{definition}[theorem]{Definition}

\numberwithin{equation}{section}

\usepackage[all]{xy}

\numberwithin{equation}{section}

\newcommand{\Cn}{C^{n}(A, B, \varepsilon)}
\newcommand{\Cm}{C^{m}(A, B, \varepsilon)}
\begin{document}

\title{BV-operators and the secondary Hochschild complex}
\author{Mamta Balodi\footnote{Department of Mathematics, Indian Institute of Science, Bangalore. Email : mamta.balodi@gmail.com} $\qquad$ Abhishek Banerjee\footnote{Department of Mathematics, Indian Institute of Science, Bangalore. Email : abhishekbanerjee1313@gmail.com} \footnote{AB was partially supported by SERB Matrics fellowship MTR/2017/000112} $\qquad$  Anita Naolekar\footnote{Stat-Math Unit, Indian Statistical Institute, Bangalore. Email:
anita@isibang.ac.in} }
\date{}
\maketitle

\begin{abstract}

We introduce the notion of a BV-operator $\Delta=\{\Delta^n:V^n\longrightarrow V^{n-1}\}_{n\geq 0}$ on a homotopy $G$-algebra $V^\bullet$ such that
the Gerstenhaber bracket on $H(V^\bullet)$ is determined by $\Delta$ in a manner similar to the BV-formalism. As an application, we produce a BV-operator on the cochain complex defining the secondary Hochschild cohomology of a symmetric algebra $A$ over a commutative algebra $B$.

\end{abstract}

{\bf MSC 2010 Subject Classification} 16E40

{\bf Keywords: } Secondary Hochschild cohomology, Gerstenhaber bracket, BV-operators

\section{Introduction}

A Gerstenhaber algebra (see \cite{Gerst}) consists of a graded vector space $W^\bullet=\bigoplus_{n\geq 0}W^n$ equipped with the following two structures:

\smallskip
(a) A dot product $x\cdot y$ of degree zero making $W^\bullet$ into an associative graded commutative algebra.

\smallskip
(b) A bracket $[x,y]$ of degree $-1$ making $W^\bullet$ into a graded Lie algebra satisfying the compatibility property that
\begin{equation*}
[x,y\cdot z]=[x,y]\cdot z +(-1)^{(deg(x)-1)deg(y)}y\cdot [x,z]
\end{equation*} Gerstenhaber algebra structures appear in a variety of situations, from Hochschild cohomology of algebras to the exterior algebra of a Lie algebra and the algebra
of differential forms on a Poisson manifold.

\smallskip
An operator $\partial=\{\partial^n:W^n\longrightarrow W^{n-1}\}_{n\geq 0}$ on $W^\bullet$ of degree $-1$ is said to generate the Gerstenhaber bracket (see Koszul 
\cite[$\S$ 2]{Kos} and also \cite[Definition 3.2]{Kosmm}) if it satisfies
\begin{equation*}
 [x,y]= (-1)^{(deg(x)-1)deg(y)}( \partial(x)\cdot y + (-1)^{deg(x)} x \cdot  \partial(y)-\partial(x\cdot y))
\end{equation*}
In particular, a Batalin-Vilkovisky algebra (or BV-algebra) consists of a Gerstenhaber algebra along with a  generator $\partial$ for the bracket such that  $\partial^2=0$.

\smallskip
In\cite{GV0}, \cite{GV}, Gerstenhaber and Voronov introduced the notion of a homotopy $G$-algebra, which is a brace algebra equipped with a differential of degree $1$ and a dot product of degree $0$ satisfying certain conditions. In particular, the cohomology groups $H(V^\bullet)$ of a homotopy $G$-algebra $V^\bullet$ carry the structure of a Gerstenhaber algebra. 

\smallskip
In this paper, we introduce the notion of a BV-operator $\Delta=\{\Delta^n:V^n\longrightarrow V^{n-1}\}_{n\geq 0}$ on a homotopy $G$-algebra $V^\bullet$ such that
the Gerstenhaber bracket on $H(V^\bullet)$ is determined by $\Delta$ in a manner similar to the BV-formalism. More explicitly, for classes 
$\bar{f}\in H^n(V^\bullet)$ and $\bar{g}\in H^m(V^\bullet)$, we have
\begin{equation*}
[\bar f,\bar g]= (-1)^{(n-1)m}\overline{( \Delta(f)\cdot g + (-1)^n f \cdot  \Delta(g)-\Delta(f \cdot g))}\textrm{ }\in H^{m+n-1}(V^\bullet)
\end{equation*} where $f\in Z^n(V^\bullet)$, $g\in Z^m(V^\bullet)$ are cocycles representing $\bar{f}$ and $\bar{g}$ respectively. We note that $\Delta$ need not be a morphism
of cochain complexes and therefore may not induce any operator on $H(V^\bullet)$. As such, $\Delta$ may not descend to a generator for the Gerstenhaber bracket on
$H(V^\bullet)$. 

\smallskip
Our motivation is to introduce a BV-operator on the cochain complex defining the secondary Hochschild cohomology of a symmetric algebra $A$ over a commutative algebra $B$. For a datum 
$(A,B,\varepsilon)$ consisting of an algebra $A$, a commutative algebra $B$ and an extension of rings $\varepsilon : B\longrightarrow A$ such that $\varepsilon(B)\subseteq Z(A)$, 
the secondary Hochschild cohomology $H^*(A,B,\varepsilon)$ was introduced by Staic \cite{S} in order to study deformations of algebras $A[[t]]$ having a $B$-algebra structure. 

\smallskip In \cite{SS}, Staic and Stancu showed that the secondary Hochschild complex $C^*(A,B,\varepsilon)$ is a non-symmetric operad with multiplication, giving it the structure of a homotopy
$G$-algebra. Hence, the secondary cohomology $H^*(A,B,\varepsilon)$ is equipped with a graded commutative cup product and a Lie bracket which makes it a Gerstenhaber algebra. For more
on the secondary cohomology, the reader may see, for instance, \cite{AB}, Corrigan-Salter and Staic \cite{CSt}, Laubacher, Staic and Stancu \cite{LSS}.

\smallskip
Let $k$ be a field. It is well known (see Tradler \cite{T}) that the Hochschild cohomology of a finite dimensional $k$-algebra $A$ equipped with a symmetric, non-degenerate, invariant bilinear form $\langle .,.\rangle : A\times A\longrightarrow k$ carries the structure of a BV-algebra. For the terms $C^n((A,B,\varepsilon))=Hom_k(A^{\otimes n}\otimes B^{\otimes \frac{n(n-1)}{2}},A)$ in the secondary Hochschild complex, we define the BV-operator
$\Delta =\sum_{i=1}^{n+1} (-1)^{in}\Delta_i: C^{n+1}(A, B, \varepsilon)\longrightarrow C^{n}(A, B, \varepsilon)$ by the condition (see Section 3)

 $$   
<\langle  \Delta_i f \left(\otimes \left(
 \begin{array}{cccccc}
 a_1& b_{1,2}& b_{1, 3}&\ldots & b_{1, n}&\\
1& a_2 &b_{2,3}&\ldots &b_{2, n}&\\
\vdots & \vdots &\vdots & \vdots & \vdots &\\
1 & 1 & 1&\ldots & b_{n-1, n}&\\
1 & 1 & 1& \ldots & a_{n}&
\end{array}
\right) \right), a_{n+1}\rangle
$$

 $$=  \langle f \left(\otimes \left(
 \begin{array}{ccccccccc}
 a_i & b_{i, i+1} & b_{i, i+2} & \cdots b_{i, n}&1 & b_{1, i} & b_{2, i} & \ldots &b_{i-1, i}\\
 1& a_{i+1} & b_{i+1, i+2} &\cdots b_{i+1, n}&1 & b_{1, i+1}& b_{2, i+1} &\ldots &b_{i-1, i+1}\\
 \vdots & \vdots &\vdots & \vdots &\vdots & \vdots&\vdots &&\vdots \\
1& 1 & \ldots & a_n &1&  b_{1, n} & b_{2, n} &\ldots & b_{i-1, n}\\
1& 1 & \ldots & 1 & a_{n+1}&  1 & \ldots & & 1\\
1 & 1 & \ldots& 1& 1 & a_1 & b_{1,2}&  \ldots & b_{1, i-1}\\
 \vdots & \vdots &\vdots &\vdots & \vdots & \vdots &\vdots &&\vdots\\
 1&1 &\ldots &\ldots &\ldots &\ldots &\ldots& & a_{i-1}
\end{array}
\right) \right), 1\rangle
$$

Then, we show that the Gerstenhaber bracket on the secondary Hochschild cohomology of $(A,B,\varepsilon)$ is determined by $\Delta$ in a manner similar to the BV-formalism.

\section{Main Result: BV-operator on homotopy $G$-algebra}

We begin by recalling the notion of a homotopy $G$-algebra from \cite{GV}. A brace algebra (see \cite[Definition 1]{GV}) is a graded vector space $V=\bigoplus_{n\geq 0}V^n$ with a collection of multilinear operators (braces) $x\{x_1,...,x_n\}$ satisfying the following conditions (with $x\{\}$ understood to be $x$):

\begin{itemize}

\item[(1)] $deg(x\{x_1,...,x_n\})=deg(x)+\sum_{i=1}^ndeg(x_i)-n$

\item[(2)] For homogeneous elements $x$, $x_1$, ..., $x_m$, $y_1$, ...,$y_n$, we have \begin{equation*}
\begin{array}{ll}
x\{x_1,...,x_m\}\{y_1,...,y_n\}&=\underset{0\leq i_1\leq j_1\leq i_2\leq ...\leq i_m\leq j_m\leq n}{\sum}(-1)^\epsilon x\{y_1,...,y_{i_1},x_1\{y_{i_1+1},...,y_{j_1}\},y_{j_1+1},...,y_{i_m},\\
& \qquad \qquad \qquad \qquad \qquad x_m\{y_{i_m+1},...,y_{j_m}\},y_{j_m+1},...,y_n\}
\end{array}
\end{equation*} where $\epsilon =\sum_{p=1}^m |x_p|(\sum_{q=1}^{i_p}|y_k|)$ and $|x|:=deg(x)-1$.
\end{itemize} 

\begin{definition} (see \cite[Definition 2]{GV}) A homotopy $G$-algebra consists of the following data:

\smallskip

(1) A brace algebra $V=\bigoplus_{n\geq 0}V^n$. 

\smallskip
(2) A dot product  of degree zero
\begin{equation*}
V^m\otimes V^n\longrightarrow V^{m+n} \qquad x\otimes y\mapsto x\cdot y 
\end{equation*} for all $m$, $n\geq 0$. 

\smallskip
(3) A differential $d:V^\bullet\longrightarrow V^{\bullet+1}$ of degree one making $V$ into a D$G$-algebra with respect to the dot product.

\smallskip
(4) The dot product satisfies the following compatibility conditions
\begin{equation*}
(x_1\cdot x_2)\{y_1,...,y_n\}=\underset{k=0}{\overset{n}{\sum}} \textrm{ }(-1)^{\epsilon_k}( x_1\{y_1,...,y_k\})\cdot (x_2\{y_{k+1},...,y_n\})
\end{equation*} where $\epsilon_k=|x_2|\sum_{p=1}^k|y_p|$ and 
\begin{equation*}
\begin{array}{l}
d(x\{x_1,...,x_{n+1}\})-(dx)\{x_1,...,x_{n+1}\}-(-1)^{|x|}\underset{i=1}{\overset{n+1}{\mathlarger{\sum}}}(-1)^{|x_1|+...+|x_{i-1}|}x\{x_1,...,dx_i,...,x_{n+1}\}\\
= (-1)^{|x||x_1|+1}x_1\cdot x\{x_2,...,x_{n+1}\}+(-1)^{|x|}\underset{i=1}{\overset{n}{\mathlarger{\sum}}}(-1)^{|x_1|+...+|x_{i-1}|}x\{x_1,...,x_i\cdot x_{i+1},...,x_{n+1}\}-x\{x_1,...,x_n\}\cdot x_{n+1}\\
\end{array}
\end{equation*}
\end{definition}

 In particular, a homotopy $G$-algebra is equipped with a graded Lie bracket which descends to the cohomology of the corresponding cochain complex $(V^\bullet,d)$ (see \cite{GV})
\begin{equation}\label{brax}
[-,-]: H^m(V^\bullet) \otimes H^n(V^\bullet) \longrightarrow H^{m+n-1}(V^\bullet)
\end{equation}  The dot product also descends to the cohomology and the bracket with an element  becomes a graded derivation for the induced dot product on $H(V^\bullet)=\bigoplus_{n\geq 0}H^n(V^\bullet)$. In other words, the cohomology  $(H(V^\bullet),[.,.],\cdot)$ of a homotopy $G$-algebra $V^\bullet$ is canonically equipped with the structure of a Gerstenhaber algebra. 

\smallskip
We now introduce the notion of a BV-operator on a homotopy $G$-algebra.

\begin{definition}\label{BVop} Let  $V^\bullet=\underset{n\geq 0}{\bigoplus} V^n$ be a homotopy $G$-algebra, let $d:V^\bullet\longrightarrow V^{\bullet+1}$ be its differential and let $[-,-]: V^n \otimes V^m\longrightarrow V^{m+n-1}$ be  its Lie bracket. We will say that a  family $\Delta=\{\Delta^n:V^n\longrightarrow V^{n-1}\}_{n\geq 0}$ is a BV-operator on $V^\bullet$ if it satisfies
\begin{equation*}
[f,g]-(-1)^{(n-1)m} ( \Delta(f)\cdot g + (-1)^n f \cdot  \Delta(g)-\Delta(f \cdot g))\in d(V^{m+n-2})
\end{equation*} for any cocycles $f\in Z^n(V^\bullet)$, $g\in Z^m(V^\bullet)$. 

\end{definition} 

If $V^\bullet$ is a homotopy $G$-algebra equipped with a BV-operator $\Delta$, we now show that the bracket on the Gerstenhaber algebra $H(V^\bullet)$ is determined by
$\Delta$ in a manner similar to the BV-formalism. 

\begin{theorem}\label{hombv} Let  $V^\bullet=\underset{n\geq 0}{\bigoplus} V^n$ be a homotopy $G$-algebra equipped with a BV-operator  $\Delta=\{\Delta^n:V^n\longrightarrow V^{n-1}\}_{n\geq 0}$. Consider $\bar f\in H^n(V^\bullet)$ and $\bar g\in H^m(V^\bullet)$ and choose cocycles $f\in Z^n(V^\bullet)$ and $g\in Z^m(V^\bullet)$ corresponding respectively
to $\bar f$ and $\bar g$. Then, we have
\begin{equation*} ( \Delta(f)\cdot g + (-1)^n f \cdot  \Delta(g)-\Delta(f \cdot g))\textrm{ }\in Z^{m+n-1}(V^\bullet)
\end{equation*} The Gerstenhaber bracket on the cohomology of $V^\bullet$ is now determined by 
\begin{equation*}
[\bar f,\bar g]= (-1)^{(n-1)m}\overline{( \Delta(f)\cdot g + (-1)^n f \cdot  \Delta(g)-\Delta(f \cdot g))}\textrm{ }\in H^{m+n-1}(V^\bullet)
\end{equation*} In particular, the right hand side does not depend on the choice of representatives $f$ and $g$. 
\end{theorem} 

\begin{proof} We know that $f\in Z^n(V^\bullet)$ and $g\in Z^m(V^\bullet)$. Since the bracket $[-,-]: V^n \otimes V^m\longrightarrow V^{m+n-1}$ descends to a bracket on the cohomology, it follows that $[f,g]\in Z^{m+n-1}(V^\bullet)$. Since  $\Delta=\{\Delta^n:V^n\longrightarrow V^{n-1}\}_{n\geq 0}$ is a BV-operator, it follows from Definition \ref{BVop} that
\begin{equation}\label{2.2equ}
[f,g]-(-1)^{(n-1)m} ( \Delta(f)\cdot g + (-1)^n f \cdot  \Delta(g)-\Delta(f \cdot g))\in d(V^{m+n-2})
\end{equation} Let us put $z_1=[f,g]$ and $z_2=(-1)^{(n-1)m} ( \Delta(f)\cdot g + (-1)^n f \cdot  \Delta(g)-\Delta(f \cdot g))$. Since $z_1-z_2\in d(V^{m+n-2})$, we must
have $z_1-z_2\in  Z^{m+n-1}(V^\bullet)$. We have already seen that $z_1\in Z^{m+n-1}(V^\bullet)$. Hence, $z_2\in Z^{m+n-1}(V^\bullet)$. By \eqref{2.2equ}, we know
that $z_1-z_2$ is a coboundary and hence the cohomology classes $\bar{z}_1=\bar{z}_2$. The result is now clear. 
\end{proof}

\section{Application : BV-operator on secondary Hochschild cohomology}

Let $k$ be a field and
 $A$ be an algebra over $k$. Let $B$ be a commutative $k$-algebra and $\varepsilon:B \longrightarrow A$ be a morphism of $k$-algebras such that $\varepsilon(B) \subseteq Z(A)$, where $Z(A)$ denotes the center of $A$. Let $M$ be an $A$-bimodule such that $\varepsilon(b)m=m\varepsilon(b)$ for all $b \in B$ and $m \in M$. Following \cite[$\S$ 4.2]{S}, we consider the complex
 $(C^\bullet((A,B,\varepsilon);M),\delta^\bullet)$ whose terms are given by
  \begin{equation*}C^n((A,B,\varepsilon);M)=Hom_k(A^{\otimes n}\otimes B^{\otimes \frac{n(n-1)}{2}},M)
\end{equation*} An element in $A^{\otimes n}\otimes B^{\otimes \frac{n(n-1)}{2}}$ will be expressed  as a ``tensor matrix'' of the form
\begin{equation*}
\bigotimes
\left(
\begin{array}{ccccccc}
 a_{1}& b_{1,2} & b_{1,3} & ...&b_{1,n-1}&b_{1,n} \\
1 & a_{2}       & b_{2,3} &...&b_{2,n-1}&b_{2,n} \\
1& 1  & a_{3} &...&b_{3,n-1}      &b_{3,n} \\
. & .       &. &...&.&. \\
1& 1& 1  &...&a_{n-1}&b_{n-1,n} \\
1 & 1& 1  &...&1&a_n\\
\end{array}
\right)
\end{equation*} where $a_i\in A$ and $b_{i,j}\in B$.
The differentials \begin{equation*}\delta^n:C^n((A,B,\varepsilon);M)\longrightarrow  C^{n+1}((A,B,\varepsilon);M)
\end{equation*} 
may be described as follows
\begin{eqnarray*}&
\delta^n(f)\left(\displaystyle\bigotimes
\left(
\begin{array}{cccccccc}
 a_{1}& b_{1,2} & b_{1,3} & ...&b_{1,n-1}&b_{1,n}&b_{1,n+1}\\
1 & a_{2}       & b_{2,3} &...&b_{2,n-1}&b_{2,n}&b_{2,n+1}\\
1& 1  & a_{3} &...&b_{3,n-1}      &b_{3,n}&b_{3,n+1}\\
. & .       &. &...&.&.&.\\
1& 1& 1  &...&1&a_{n}&b_{n,n+1}\\
1 & 1& 1  &...&1&1&a_{n+1}\\
\end{array}
\right)\right)=
\end{eqnarray*}
\begin{eqnarray*}
&
a_1\varepsilon(b_{1,2}b_{1,3}...b_{1,n+1})f\left(\displaystyle\bigotimes
\left(
\begin{array}{cccccc}
 a_{2}       & b_{2,2} &...&b_{2,n}&b_{2,n+1}\\
1  & a_{3} &...    &b_{3,n}&b_{3,n+1}\\
 .       &. &...&.&.\\
1& 1  &...&a_{n}&b_{n,n+1}\\
 1& 1  &...&1&a_{n+1}\\
\end{array}
\right)\right)+&\\
&\;&\;\\
& \underset{i=1}{\overset{n}{\sum}} (-1)^{i}f\left(\displaystyle\bigotimes
\left(
\begin{array}{cccccccc}
 a_{1}       & b_{1,2}&... &b_{1,i}b_{1,i+1}&...&b_{1,n}&b_{1,n+1}\\
1  &a_{2}&...&b_{2,i}b_{2,i+1}     &...&b_{2,n}&b_{2,n+1}\\
 .       &. &...&...&...&.&.\\
  1       &1 &...&\varepsilon(b_{i,i+1})a_{i}a_{i+1}&...&b_{i,n}b_{i+1,n}&b_{i,n+1}b_{i+1,n+1}\\
   .       &. &...&...&...&.&.\\
1& 1  &...&...&...&a_{n}&b_{n,n+1}\\
 1& 1  &...&...&...&1&a_{n+1}\\
\end{array}
\right)\right)+&\\
&(-1)^{n+1}f\left(\displaystyle\bigotimes
\left(
\begin{array}{cccccc}
 a_{1}       & b_{1,2} &...&b_{1,n-1}&b_{1,n}\\
1  &a_{2} &...     &b_{2,n-1}&b_{2,n}\\
 .       &. &...&.&.\\
1& 1  &...&a_{n-1}&b_{n-1,n}\\
 1& 1  &...&1&a_{n}\\
\end{array}
\right)\right)\varepsilon(b_{1,n+1}b_{2,n+1}...b_{n,n+1})a_{n+1}&
\end{eqnarray*}
for $f \in C^n((A,B,\varepsilon);M),$ $a_i \in A,$ $b_{i,j} \in B$. The cohomology groups  of   $(C^\bullet((A,B,\varepsilon);M),\delta^\bullet)$  are known as the secondary Hochschild cohomologies $H^n((A, B, \varepsilon); M)$  of the triple $(A,B,\varepsilon)$ with coefficients in $M$ (see \cite{S}). 

\smallskip From \cite[Proposition 3.1]{SS}, we know that the secondary Hochschild complex $C^\bullet(A,B,\varepsilon):=C^\bullet((A,B,\varepsilon);A)$ carries the structure of a homotopy $G$-algebra. This induces a graded Lie bracket 
\begin{equation}
[-,-]: H^m(A,B,\varepsilon) \otimes H^n(A,B,\varepsilon) \longrightarrow H^{m+n-1}(A,B,\varepsilon)
\end{equation} on the secondary cohomology. It follows (see \cite[Corollary 3.2]{SS}) that the secondary cohomology $H^\bullet(A,B,\varepsilon)$ 
carries the structure of a Gerstenhaber algebra in the sense of  \cite{Gerst}.

\smallskip
From now onwards, we always let $A$ be a finite dimensional $k$-algebra equipped with a symmetric, non-degenerate, invariant bilinear form $\langle \cdot,\cdot \rangle: A \times A \longrightarrow k$. In particular,  $\langle a_1,a_2 \rangle=\langle a_2,a_1 \rangle$, $\langle a_1a_2,a_3 \rangle=\langle a_1,a_2a_3 \rangle$ for any $a_1,a_2,a_3 \in A$.  
For $i \in \{1, \ldots, n+1\}$, we define the maps
$\Delta_i: C^{n+1}(A, B, \varepsilon)\longrightarrow C^{n}(A, B, \varepsilon)$ as follows:
 $$   
<\langle  \Delta_i f \left(\otimes \left(
 \begin{array}{cccccc}
 a_1& b_{1,2}& b_{1, 3}&\ldots & b_{1, n}&\\
1& a_2 &b_{2,3}&\ldots &b_{2, n}&\\
\vdots & \vdots &\vdots & \vdots & \vdots &\\
1 & 1 & 1&\ldots & b_{n-1, n}&\\
1 & 1 & 1& \ldots & a_{n}&
\end{array}
\right) \right), a_{n+1}\rangle
$$

 $$=  \langle f \left(\otimes \left(
 \begin{array}{ccccccccc}
 a_i & b_{i, i+1} & b_{i, i+2} & \cdots b_{i, n}&1 & b_{1, i} & b_{2, i} & \ldots &b_{i-1, i}\\
 1& a_{i+1} & b_{i+1, i+2} &\cdots b_{i+1, n}&1 & b_{1, i+1}& b_{2, i+1} &\ldots &b_{i-1, i+1}\\
 \vdots & \vdots &\vdots & \vdots &\vdots & \vdots&\vdots &&\vdots \\
1& 1 & \ldots & a_n &1&  b_{1, n} & b_{2, n} &\ldots & b_{i-1, n}\\
1& 1 & \ldots & 1 & a_{n+1}&  1 & \ldots & & 1\\
1 & 1 & \ldots& 1& 1 & a_1 & b_{1,2}&  \ldots & b_{1, i-1}\\
 \vdots & \vdots &\vdots &\vdots & \vdots & \vdots &\vdots &&\vdots\\
 1&1 &\ldots &\ldots &\ldots &\ldots &\ldots& & a_{i-1}
\end{array}
\right) \right), 1\rangle
$$

To clarify the above operator, let us express 
\begin{equation*}
\bigotimes \left(
 \begin{array}{cccccc}
 a_1& b_{1,2}& b_{1, 3}&\ldots & b_{1, n}&\\
1& a_2 &b_{2,3}&\ldots &b_{2, n}&\\
\vdots & \vdots &\vdots & \vdots & \vdots &\\
1 & 1 & 1&\ldots & b_{n-1, n}&\\
1 & 1 & 1& \ldots & a_{n}&
\end{array}
\right) =\begin{pmatrix} 
U(i-1) & X_{12} \\
1 & U(n-i-1) \\
\end{pmatrix}
\end{equation*} where $U(k)$ is a square matrix of dimension $k$. Then, we have
\begin{equation*}
\langle  \Delta_i f\begin{pmatrix} 
U(i-1) & X_{12} \\
1 & U(n-i-1) \\
\end{pmatrix}, a_{n+1}\rangle =\langle f\begin{pmatrix}
U(n-i-1) & 1 & X_{12}^t \\
1& a_{n+1} & 1 \\
1 & 1 & U(i) \\
\end{pmatrix} ,1 \rangle
\end{equation*} where $X_{12}^t$ denotes the transpose of $X_{12}$. The operator $\Delta: C^{n+1}(A, B, \varepsilon)\longrightarrow C^{n}(A, B, \varepsilon)$ is then defined as
\begin{equation*}\Delta:= \sum_{i=1}^{n+1} (-1)^{in} \Delta_i.
\end{equation*}
  Following \cite[$\S$ 3]{SS}, we know that the complex $C^\bullet(A,B,\varepsilon)$ carries a   dot product of degree $0$, i.e., for 
  $f\in C^n(A,B,\varepsilon)$, $g\in C^m(A,B,\varepsilon)$, we have $f\cdot g\in C^{m+n}(A,B,\varepsilon)$. We also consider the operations
  \begin{equation*}
  \circ_i: C^n(A,B,\varepsilon)\otimes C^m(A,B,\varepsilon)\longrightarrow C^{m+n-1}(A,B,\varepsilon)
  \end{equation*} and set $f\circ g:=\sum_{i=1}^n(-1)^{(i-1)(m-1)}f\circ_ig$ as in \cite[$\S$ 3]{SS}. We also set 
  \begin{equation*}
  \begin{array}{c} \rho^1, \rho^2: C^{n}(A, B, \varepsilon) \otimes C^{m}(A, B, \varepsilon)\longrightarrow C^{n+m-1}(A, B, \varepsilon)\\ 
  \rho^1(f \otimes g) := \sum_{i=1}^m (-1)^{i(n+m-1)}\Delta_i(f\cdot g)\qquad 
\rho^2(f \otimes g) := \sum_{i=m+1}^{m+n} (-1)^{i(n+m-1)}\Delta_i(f\cdot g)\\
\end{array}
\end{equation*}    for $f\in C^n(A,B,\varepsilon)$, $g\in C^m(A,B,\varepsilon)$. It is clear that $\rho^1(f\otimes g)+\rho^2(f\otimes g)=\Delta(f\cdot g)$.

\begin{lemma}\label{delta}
$\rho^1(f \otimes g)=(-1)^{nm} \rho^2(g \otimes f)$ for all $f \in \Cn$ and $g \in \Cm$.
\end{lemma}

\begin{proof}
This may  be verified by direct computation.
\end{proof}

\begin{lemma}\label{H}
Let $f\in Z^{n}(A, B, \varepsilon) $, $g \in Z^{m}(A, B, \varepsilon) $. Then $f\circ g -(-1)^{(n-1)m} \Delta(f)\cdot g + (-1)^{(n-1)m} \rho^2(f\otimes g)$ is a coboundary. In fact, if we define $H$
\begin{equation*} H= \sum_{i, j \geq 1, i+j\leq n} (-1)^{(j-1)(m-1) + i(n+m)+1} \Delta_i(f\circ_j g),\end{equation*}
then, \begin{equation*}
\delta H= f\circ g -(-1)^{(n-1)m} \Delta(f) \cdot g + (-1)^{(n-1)m} \rho^2(f\otimes g).
\end{equation*}
\end{lemma}

\begin{proof}
We set, for $k\geq 0$, $p\geq 0$:
\begin{equation*}
T^k_{k+p}=\bigotimes \begin{pmatrix}
a_{k+1} & \dots & b_{k+1,k+p}\\
\vdots & & \vdots \\
1 & \dots & a_{k+p}
\end{pmatrix}
\end{equation*}
We see that

 \begin{equation}\label{delDel}\small 
\langle \delta(\Delta_i (f \circ_j g))\left(\otimes \left(
 \begin{array}{ccccccc} 
 a_1& b_{1,2}& b_{1, 3}&\ldots & b_{1, n+m-1}&\\
1& a_2 &b_{2,3}&\ldots &b_{2, n+m-1}&\\
\vdots & \vdots &\vdots & \vdots & \vdots &\\
1 & 1 & 1&\ldots & b_{n+m-2, n+m-1}&\\
1 & 1 & 1& \ldots & a_{n+m-1}&
\end{array}
\right) \right), a_{n+m}\rangle=
\end{equation}
\begin{equation*}\tiny \langle f \left(
 \begin{array}{ccccccccc} 
a_{i+1} & \ldots & b_{i+1, i+j-1} & \prod \limits_{k=0}^{m-1} b_{i+1, i+j+k} & \ldots & 1&b_{2, i+1} &\ldots &b_{i-1, i+1}\\
\vdots &\vdots & \vdots & \vdots &\vdots && \vdots& \vdots\\
1 &  \ldots & a_{i+j-1} &  \prod \limits_{k=0}^{m-1} b_{i+j-1, i+j+k}& \ldots & 1 & b_{2, i+j-1} &\ldots &b_{i-1, i+j-1}\\
1 &  \ldots & 1 & g(T^{i+j-1}_{i+j+m-1})& \prod\limits_{k=0}^{m-1} b_{i+j+k, i+j+m}  &1 &\ldots  &\ldots &\prod \limits_{k=0}^{m-1} b_{i-1, i+j+k}\\
\vdots &\vdots & \vdots & \vdots & \vdots &\vdots\\
1 & \ldots & 1 &  1&  1 & \alpha & 1 & \ldots &1\\
1 & \ldots & 1 &  1&  1 & 1  & a_2   &\ldots & b_{2, i}\\
\vdots &\vdots & \vdots & \vdots &\vdots && \vdots & \vdots\\
1 & \ldots & 1 &  1&  1 & 1  & 1 &\ldots & a_i\\
\end{array}
\right), 1\rangle
\end{equation*} 
\begin{equation*}+\sum_{\lambda=1}^{i-1}(-1)^\lambda~  
\tiny \langle f\left(
 \begin{array}{cccccccccc} 
a_{i+1}& \ldots & &\prod\limits_{k=0}^{m-1}b_{i+1, i+j+k}& \ldots &1& b_{1, i+1} & b_{\lambda,i+1}b_{\lambda+1,i+1} & b_{i, i+1}\\
\vdots & \vdots & & \vdots & \vdots &\vdots & \vdots & \vdots&  \vdots\\
1& \cdots &  & \prod\limits_{k=0}^{m-1}b_{i+j-1, i+j+k}& \ldots & 1 & b_{1, i+j-1} & b_{\lambda, i+j-1} b_{\lambda+1, i+j-1} & b_{i, i+j-1}\\
1 & \cdots & & g(T^{i+j-1}_{i+j+m-1}) &  \prod\limits_{k=0}^{m-1} b_{i+j+k, i+j+m}&  \ldots & \prod\limits_{k=0}^{m-1} b_{1, i+j+k}&  \prod\limits_{k=0}^{m-1} b_{\lambda,i+j+k} b_{\lambda+1,i+j+k} &\prod\limits_{k=0}^{m-1} b_{i, i+j+k}\\
\vdots & \vdots && \vdots & \vdots &1 & \vdots & \vdots&  \vdots\\
1 & 1 & & \ldots & \ldots & a_{n+m}& 1 & \ldots & 1 \\
1 & 1 &  & \ldots & \ldots & 1& a_1 & \ldots & b_{1. i} \\
\vdots & \vdots && \vdots & \vdots &\vdots & \vdots & \vdots&  \vdots\\
1 & 1 &\ldots & 1 & 1 & 1 & 1 & \beta_\lambda& \ldots b_{\lambda, i}b_{\lambda+1, i} \\
\vdots & \vdots && \vdots & \vdots &\vdots & \vdots & \vdots&  \vdots\\
1 & 1 & \ldots & 1 &1 & 1& 1 & 1 &a_i
\end{array}
 \right), 1\rangle
\end{equation*}

$$ +\sum_{\lambda=i}^{i+j-2}(-1)^\lambda~ \tiny \langle f\left(
 \begin{array}{ccccccccccc} 
a_i & b_{i,i+1}& b_{i,\lambda}b_{i,\lambda+1} &\prod\limits_{k=0}^{m-1}b_{i, i+j+k} & \ldots &1 & b_{1,i} &\ldots &b_{i-1, i}\\
1 & a_{i+1} &b_{i+1,\lambda}b_{i+1,\lambda+1}   &\prod\limits_{k=0}^{m-1}b_{i+1, i+j+k} & \ldots &1& b_{1,i+1} &\ldots & b_{i-1, i+1}\\
1 & 1 &\vdots& \vdots &\vdots &&\vdots&  \vdots &\vdots \\
1 &  \ldots &\beta_\lambda&\prod\limits_{k=0}^{m-1}b_{\lambda, i+j+k}b_{\lambda+1, i+j+k} &\ldots &1 &b_{1, \lambda}  &\ldots & b_{i-1, \lambda}\\
\vdots & \vdots& \vdots & \vdots & \vdots & \vdots&  \vdots\\
1 & \ldots& &\prod\limits_{k=0}^{m-1}b_{i+j-1, i+j+k} & \ldots&1&b_{1, i+j-1} &   & b_{i-1, i+j-1}\\
1 & \ldots& \ldots & g(T^{i+j-1}_{i+j+m-1}) &  \prod\limits_{k=0}^{m-1}b_{i+j+k, i+j+m} &1 & \ldots & \ldots &\prod\limits_{k=0}^{m-1}b_{i-1,i+j+k}\\
1 & \ldots & \ldots & 1 & 1 & a_{m+n} & 1 & \ldots &1 \\
1 & \ldots & \ldots & 1 & 1 & 1 & a_1 & \ldots &b_{1,i-1}\\
\vdots & \vdots&\vdots &\vdots & \vdots & && \vdots & \vdots& \\
1 & \ldots& \ldots &1&\ldots & 1 & 1 &  \ldots &a_{i-1} \\
\end{array}
 \right), 1\rangle
$$

$$  +\sum_{\lambda=i+j-1}^{i+j+m-2}(-1)^\lambda~ \tiny \langle f\left(
 \begin{array}{ccccccccccc} 
a_i & \ldots & \prod\limits_{k=-1}^{m-2}b_{i, i+j+k} & \ldots & 1 & b_{1, i}& \ldots & b_{i-1, i}\\
\vdots & \vdots &\vdots & \vdots & \vdots & \vdots & \vdots &  \vdots\\
1 & a_{i+j-2}& \prod\limits_{k=-1}^{m-2}b_{i+j-2, i+j+k} & \ldots & 1 & b_{1, i+j-2} &\ldots & b_{i-1, i+j-2}\\
1 & 1 & g(\bar{T}^{i+j-2}_{i+j+m-2}) & \prod\limits_{k=-1}^{m-2} b_{i+j+k, i+j+m-1}& 1 &\ldots & \ldots & \prod\limits_{k=-1}^{m-2}b_{i-1,i+j+k}\\
\vdots & \vdots &\vdots & \vdots & 1 & \vdots & \vdots &  \vdots\\
1 & 1 & 1 & \ldots & a_{n+m} & 1 & 1 & \ldots   1\\
1 & 1 & 1 & \ldots & 1& a_1& b_{1, 2} & \ldots  b_{1, i-1}\\
\vdots & \vdots &\vdots & \vdots & 1 & \vdots & \vdots \\
1 & \ldots &1 & \ldots & 1& 1&\ldots&  a_{i-1}\\
\end{array}
 \right), 1\rangle
$$

$$+\sum_{\lambda=i+j+m-1}^{n+m-2}(-1)^\lambda~    \tiny \langle f\left(
 \begin{array}{ccccccccccc} 
a_i & b_{i, i+1}& \ldots &b_{i, i+j-2} & \prod\limits_{k=-1}^{m-2} b_{i, i+j+k} & \ldots  &  & 1 &\ldots & b_{i-1, i}\\
\vdots & \vdots &\vdots & \vdots & \vdots& \vdots &  &    1 & \vdots &  \vdots\\
1& \ldots & \ldots &a_{i+j-2}& \prod\limits_{k=-1}^{m-2} b_{i+j-2, i+j+k} & \ldots &  &1 & \ldots & b_{i-1, i+j-2}\\
1 & \ldots & \ldots  &g(T^{i+j-2}_{i+j+m-2}) &  \prod\limits_{k=-1}^{m-2} b_{i+j+k, i+j+m-1}& \ldots && 1& \ldots& \prod\limits_{k=-1}^{m-2} b_{i-1, i+j+k}\\
1 & \ldots & \ldots & 1 &1 &  \beta_\lambda & & 1& \ldots & b_{i-1,\lambda}b_{i-1, \lambda+1}\\
\vdots & \vdots &\vdots & \vdots & \vdots& \vdots &  &    1 & \vdots &  \vdots\\
1 & \ldots & \ldots & \ldots & \ldots & \ldots && a_{n+m}& 1 \ldots &1\\
1 & 1 & 1 & \ldots & \ldots  &\ldots & &\ldots & a_1 & \ldots  b_{1, i-1}\\
\vdots & \vdots &\vdots & \vdots & \vdots& \vdots &  &   \vdots& \vdots &  \vdots\\
1 & 1 & 1 & \ldots & \ldots  &\ldots & &\ldots & 1 & \ldots a_{i-1}\\

\end{array}
 \right), 1\rangle
$$

$$+(-1)^{n+m-1}~  \tiny \langle f\left(
 \begin{array}{ccccccccccc} 
a_i & b_{i, i+1}& \ldots &b_{i, i+j-2} & \prod\limits_{k=-1}^{m-2} b_{i, i+j+k} & \ldots  &  & 1 &\ldots & b_{i-1, i}\\
\vdots & \vdots &\vdots & \vdots & \vdots& \vdots &  &    1 & \vdots &  \vdots\\
1& \ldots & \ldots &a_{i+j-2}& \prod\limits_{k=-1}^{m-2} b_{i+j-2, i+j+k} & \ldots &  &1 & \ldots & b_{i-1, i+j-2}\\
1 & \ldots & \ldots  &g(T^{i+j-2}_{i+j+m-2}) &  \prod\limits_{k=-1}^{m-2} b_{i+j+k, i+j+m-1}& \ldots && 1& \ldots& \prod\limits_{k=-1}^{m-2} b_{i-1, i+j+k}\\
\vdots & \vdots &\vdots & \vdots & \vdots& \vdots &  &    1 & \vdots &  \vdots\\
1 & \ldots & \ldots & \ldots & \ldots & \ldots &&\gamma& 1 &\ldots 1\\
1 & \ldots & \ldots & \ldots & \ldots  &\ldots & &\ldots & a_1 & \ldots  b_{1, i-1}\\
\vdots & \vdots &\vdots & \vdots & \vdots& \vdots &  &   \vdots& \vdots &  \vdots\\
1 & \ldots& \ldots& \ldots & \ldots  &\ldots & &\ldots & 1 & \ldots a_{i-1}\\

\end{array}
 \right), 1\rangle
$$

where $\alpha:= \varepsilon(b_{1, 2}\cdots b_{1, n+m-1})a_{n+m}a_1,$  $\gamma:=\varepsilon(b_{1, n+m-1} \cdots b_{n+m-2, n+m-1}) a_{n+m} a_{n+m-1}$, $\beta_\lambda:= \varepsilon(b_{\lambda, \lambda+1}) a_\lambda a_{\lambda+1}$ for $1 \leq \lambda \leq n+m-2$ and 
$$\bar{T}^{i+j-2}_{i+j+m-2}:=\left(
 \begin{array}{cccccccc}
a_{i+j-1} & b_{i+j-1, i+j} &\ldots & b_{i+j-1,\lambda}b_{i+j-1,\lambda+1}&& b_{i+j-1, i+j+m-2}\\
1 & a_{i+j} & \ldots&b_{i+j,\lambda}b_{i+j,\lambda+1} & &  b_{i+j, i+j+m-2}\\
1 &1 &\vdots & & & &\\
1&1&\ldots &\varepsilon(b_{\lambda, \lambda+1})a_\lambda a_{\lambda+1}& \ldots &b_{\lambda,i+j+m-2}b_{\lambda +1,i+j+m-2}\\
\vdots &&\vdots& \vdots & \vdots & \vdots\\
\ldots & &\ldots&\ldots & \ldots & a_{i+j+m-2}
\end{array}
 \right).$$

We write the entire expression of \ref{delDel} as $$\langle \delta(\Delta_i (f \circ_j g))\left(\otimes \left(
 \begin{array}{ccccccc} 
 a_1& b_{1,2}& b_{1, 3}&\ldots & b_{1, n+m-1}&\\
1& a_2 &b_{2,3}&\ldots &b_{2, n+m-1}&\\
\vdots & \vdots &\vdots & \vdots & \vdots &\\
1 & 1 & 1&\ldots & b_{n+m-2, n+m-1}&\\
1 & 1 & 1& \ldots & a_{n+m-1}&
\end{array}
\right) \right), a_{n+m}\rangle=
E_1 + E_2 + E_3 + E_4 + E_5 + E_6,$$ 
where $E_k$ denotes the $k$-th term in the expression.

\smallskip
We set for $i$, $j\geq 1$ and $i+j\leq n$,
$$
A_{i,j}:= \tiny{(-1)^{i+1}\langle
a_i' f \tiny \left(
\begin{array}{ccccccccccc}
a_{i+1} &\ldots & b_{i+1, i+j-1}& \prod\limits_{k=0}^{m-1} b_{i+1, i+j+k} & \ldots    && 1 &\ldots &\ldots & b_{i-1, i+1}\\
\vdots & \vdots & \vdots & \vdots& \vdots &  &    1 & \vdots & \ldots &  \vdots\\
1& \ldots & a_{i+j-1}& \prod\limits_{k=0}^{m-1} b_{i+j-1, i+j+k} & \ldots &  &1 & \ldots & \ldots &b_{i-1, i+j-1}\\
1 & \ldots  &&g(T^{i+j-1}_{i+j+m-1}) &  \prod\limits_{k=0}^{m-1} b_{i+j+k, i+j+m}& \ldots & 1& \ldots& \ldots &\prod\limits_{k=0}^{m-1} b_{i-1, i+j+k}\\
\vdots & \vdots & \vdots & \vdots& \vdots &  &    1 & \vdots &  \ldots &\vdots\\
1 & \ldots & \ldots & \ldots & \ldots && a_{n+m}& 1  &\ldots &1\\
1 & 1 & \ldots & \ldots  &\ldots & &\ldots & a_1 &   \ldots &b_{1, i-1}\\
\vdots & \vdots & \vdots & \vdots& \vdots &  &   \vdots& \vdots & \ldots & \vdots\\
1 & 1 & \ldots & \ldots  &\ldots & &\ldots & 1 &  \ldots &a_{i-1}\\
\end{array}
 \right), 1\rangle}
$$ 

 $$+ E_3 +(-1)^{i+j-1} \tiny \langle f \left(
\begin{array}{cccccccccc}
a_{i} &\ldots & b_{i, i+j-2}& \prod\limits_{k=-1}^{m-2} b_{i, i+j+k} & \ldots    && 1 &\ldots & \ldots & b_{i-1, i}\\
\vdots & \vdots & \vdots & \vdots& \vdots &  &    1 & \vdots &  \ldots &\vdots\\
1& \ldots &a_{i+j-2}& \prod\limits_{k=-1}^{m-2} b_{i+j-2, i+j+k} & \ldots &  &1 & \ldots & \ldots &b_{i-1, i+j-2}\\
1 & \ldots   &\ldots &\eta &  \prod\limits_{k=0}^{m-1} b_{i+j+k, i+j+m}& \ldots & 1& \ldots& \ldots &\prod\limits_{k=0}^{m-1} b_{i-1, i+j+k}\\
\vdots & \vdots &\vdots & \vdots & \vdots& \vdots & 1 & \vdots & \ldots & \vdots\\
1 & \ldots & \ldots & \ldots & \ldots & \ldots & a_{n+m}& 1  &\ldots &1\\
1 & 1 & 1 & \ldots & \ldots  &\ldots & & a_1 & \ldots &  b_{1, i-1}\\
\vdots & \vdots &\vdots & \vdots & \vdots& \vdots &  &   \vdots& \vdots & \vdots \\
1 & 1 & 1 & \ldots & \ldots  &\ldots & &\ldots  &   \ldots &a_{i-1}\\
\end{array}
 \right), 1\rangle
$$

where
$a_i'= a_i \varepsilon(b_{i, i+1}\cdots b_{i, n+m-1}b_{1,i}\cdots b_{i-1,i}),$. We also set

$$B_{i, j}:= (-1)^{i+j+m-2} f \tiny \left(
\begin{array}{cccccccccc}
a_{i} &\ldots & b_{i, i+j-2} & \prod\limits_{k=-1}^{m-2} b_{i, i+j+k} & \ldots    && 1 &\ldots & \ldots &b_{i-1, i}\\
\vdots & \vdots & \vdots & \vdots& \vdots &  &    1 & \vdots &  \ldots &\vdots\\
1& \ldots &a_{i+j-2}& \prod\limits_{k=-1}^{m-2} b_{i+j-2, i+j+k} & \ldots &  &1 & \ldots &\ldots & b_{i-1, i+j-2}\\
1 & \ldots  &\ldots  &\zeta &  \prod\limits_{k=-1}^{m-2} b_{i+j+k, i+j+m-1}& \ldots & 1& \ldots& \ldots &\prod\limits_{k=-1}^{m-2} b_{i-1, i+j+k}\\
\vdots & \vdots &\vdots & \vdots & \vdots& \vdots & 1 & \vdots & \ldots & \vdots\\
1 & \ldots & \ldots & \ldots & \ldots & \ldots &a_{n+m}& 1  &\ldots &1\\
1 & 1 & 1 & \ldots & \ldots  &\ldots &\ldots & a_1 & \ldots &b_{1, i-1}\\
\vdots & \vdots &\vdots & \vdots & \vdots& \vdots &  &   \vdots& \vdots &  \vdots\\
1 & 1 & 1 & \ldots & \ldots  &\ldots & &\ldots & \ldots &a_{i-1}\\
\end{array}
 \right), 1\rangle +E_5+E_6
$$

where $$\eta=\tiny{a_{i+j-1} \varepsilon(b_{i+j-1, i+j} \cdots b_{i+j-1, n+m-1})} g\left(T^{i+j-1}_{i+j+m-1}\right)
$$ 
$$\zeta= g\left(T^{i+j-2}_{i+j+m-2}\right) \varepsilon\left(\prod\limits_{k=-1}^{m-2}b_{i+j+k, i+j+m-1}\right),
$$  and 

\smallskip
$$C_{i, j}:= (-1)^{i}\tiny \langle f \left(
\begin{array}{cccccccccc}
a_{i+1} &\ldots & b_{i+1, i+j-1}& \prod\limits_{k=0}^{m-1} b_{i+1, i+j+k} & \ldots    && 1 &\ldots & b_{i-1, i+1}\\
\vdots & \vdots & \vdots & \vdots& \vdots &  &    1 & \vdots &  \vdots\\
1& \ldots &a_{i+j-1}& \prod\limits_{k=0}^{m-1} b_{i+j-1, i+j+k} & \ldots &  &1 & \ldots & b_{i-1, i+j-1}\\
1 & \ldots  &\ldots&g(T^{i+j-1}_{i+j+m-1}) &  \prod\limits_{k=0}^{m-1} b_{i+j+k, i+j+m}& \ldots & 1& \ldots& \prod\limits_{k=0}^{m-1} b_{i-1, i+j+k}\\
\vdots & \vdots &\vdots & \vdots& \vdots &  &    1 & \vdots &  \vdots\\
1 & \ldots & \ldots & \ldots & \ldots && a_{n+m}& 1  &1\\
1 & 1 & \ldots & \ldots  &\ldots & &\ldots & a_1 &   b_{1, i-1}\\
\vdots & \vdots& \vdots & \vdots& \vdots &  &   \vdots& \vdots &  \vdots\\
1 & 1 & \ldots & \ldots  &\ldots & &\ldots & 1 &  a_{i-1}\\
\end{array}
 \right) a_i', 1\rangle +E_1 +E_2
$$

The first term of $A_{i,j}$ and that of $C_{i,j}$ are the same modulo a sign. Using the fact that $\delta g=0$, the third term of $A_{i,j}$ and the first term of $B_{i,j}$ add up to give $E_4$. Thus, we have
 \begin{equation}\label{ABC0}
\langle  \delta(\Delta_i (f \circ_j g))\left(\otimes \left(
 \begin{array}{ccccccc} 
 a_1& b_{1,2}& b_{1, 3}&\ldots & b_{1, n+m-1}&\\
1& a_2 &b_{2,3}&\ldots &b_{2, n+m-1}&\\
\vdots & \vdots &\vdots & \vdots & \vdots &\\
1 & 1 & 1&\ldots & b_{n+m-2, n+m-1}&\\
1 & 1 & 1& \ldots & a_{n+m-1}
\end{array}
\right) \right), a_{n+m} \rangle= A_{i,j} + B_{i,j} + C_{i,j}
\end{equation}

It may be verified that

$(-1)^{i+1}A_{i, j-1} + (-1)^{i+m}B_{i, j}+ (-1)^{i+n}C_{i-1, j}=$
\begin{equation}\label{ABC1}
<\Delta_i ((\delta f) \circ_j g))\left(\otimes \left(\\
\begin{array}{cccccc} 
 a_1& b_{1,2}& b_{1, 3}&\ldots & b_{1, n+m-1}&\\
1& a_2 &b_{2,3}&\ldots &b_{2, n+m-1}&\\
\vdots & \vdots &\vdots & \vdots & \vdots &\\
1 & 1 & 1& \ldots & a_{m+n-1}&
\end{array}
\right) \right), a_{n+m}>\\
=0
\end{equation}
for $2 \leq i \leq n$, $2 \leq j \leq n-1$ and $i+j \leq n$. The second equality in \eqref{ABC1} uses the fact that $\delta f=0$.

For $i, j \in \{1, \ldots , n+1\}$, we define
$$
A_{i, 0} :=  (-1)^{i+1}\langle g\left(T^{i-1}_{i+m-1}
\right)
f\left(
\begin{array}{cccccccccc}
 a_{i+m} & b_{i+m, i+m+1}  & \cdots b_{i+m, n+m-1}&1 & b_{1, i+m} & b_{2, i+m} \cdots &b_{i-1, i+m}\\
 1& a_{i+m+1}  &\cdots b_{i+m+1, n+m-1}&1 & b_{2, i+m+1}& b_{2, i+m+1} \ldots &b_{i-1, i+m+1}\\
 \vdots & \vdots & \vdots &\vdots & \vdots&\vdots &\vdots \\
1& 1 & \ldots & a_{n+m}&  1 & \ldots & 1\\
1 & 1 & \ldots& 1& a_1 & b_{1,2}  \ldots & b_{1, i-1}\\
 \vdots & \vdots &\vdots &\vdots & \vdots & \vdots &\vdots \\
 1&1 &\ldots &\ldots &\ldots &\ldots & a_{i-1}
\end{array}
\right) , 1\rangle,
$$

$$ C_{0, j}= \langle f\left(
\begin{array}{cccccccccc} 
a_1 & \ldots & b_{1, j-1} & \prod\limits_{k=0}^{m-1}b_{1, j+k}  & b_{1, j+m}& \ldots & b_{1, n+m-1}\\
 \vdots & \vdots & \vdots &\vdots &\vdots  & \vdots & \vdots \\
1 & \ldots & a_{j-1} & \prod\limits_{k=0}^{m-1} b_{j-1, j+k} & b_{j-1, j+m}& \ldots& b_{j-1, n+m-1}\\
1 & \ldots & 1& g(T^{j-1}_{j+m-1}) & \prod\limits_{k=0}^{m-1}b_{j+k, j+m} & \ldots &\prod\limits_{k=0}^{m-1}b_{j+k, n+m-1} \\
1 & \ldots & 1 & 1 & a_{j+m} & \ldots& b_{j+m, n+m-1}\\
 \vdots & \vdots & \vdots &\vdots &\vdots  & \vdots & \vdots \\
1 & \ldots & \ldots & \ldots & \ldots & \ldots & a_{n+m-1}
\end{array}
\right) a_{n+m}, 1\rangle,
$$
 and for $i \in \{1, \ldots, n\}$, define

$$B_{i, n-i+1} :=  (-1)^{n+m+1} \langle \tiny f\left(
\begin{array}{cccccccccc} 
a_{i} &\ldots & b_{i, n-1}& \prod\limits_{k=1}^{m} b_{i, n-1+k} & b_{1, i}& \ldots &b_{i-1, i}\\
 \vdots & \vdots & \vdots&\vdots &\vdots & \vdots & \vdots\\
1 & \ldots &a_{n-1}&  \prod\limits_{k=1}^{m} b_{n-1, n-1+k} & b_{1, n-1} & \ldots &b_{i-1, n-1}\\
1 & \ldots & 1 & g(T^{n-1}_{n+m-1})\cdot a_{n+m} & \prod\limits_{k=0}^{m-1} b_{1, n+k} & \ldots &  \prod\limits_{k=0}^{m-1} b_{i-1, n+k}\\
1 & \ldots & 1 & 1 & a_1& \ldots & b_{1, i-1}\\
 \vdots & \vdots & \vdots&\vdots &\vdots & \vdots & \vdots\\
1 & \ldots & \ldots & \ldots & \ldots & \ldots & a_{i-1}

\end{array}
\right), 1\rangle.
$$

Thus, $A_{i, j}$, $B_{i, j+1}$, $C_{i-1, j}$ are defined for all the values of $i, j$ with $i,j\geq 1$ and $i+j \leq n+1$. Moreover, it may be verified that

$A_{1,j-1}+ (-1)^{m+1} B_{1,j} + (-1)^{n+1}C_{0,j}=$
$$ \langle (\delta f)\left(
\begin{array}{cccccccccc} 
a_1 & \ldots & b_{1, j-1} & \prod\limits_{k=0}^{m-1}b_{1, j+k}  & b_{1, j+m}& \ldots & b_{1, n+m-1}&1\\
 \vdots & \vdots & \vdots &\vdots &\vdots  & \vdots & \vdots&1 \\
1 & \ldots & a_{j-1} & \prod\limits_{k=0}^{m-1} b_{j-1, j+k} & b_{j-1, j+m}& \ldots& b_{j-1, n+m-1}&1 \\
1 & \ldots & 1& g(T^{j-1}_{j+m-1}) & \prod\limits_{k=0}^{m-1}b_{j+k, j+m} & \ldots &\prod\limits_{k=0}^{m-1}b_{j+k, n+m-1}&1 \\
1 & \ldots & 1 & 1 & a_{j+m} & \ldots& b_{j+m, n+m-1}&1\\
 \vdots & \vdots & \vdots &\vdots &\vdots  & \vdots & \vdots&1 \\
1 & \ldots & \ldots & \ldots & \ldots & \ldots & a_{n+m-1}&1\\
1& \ldots & 1 & 1 & 1& 1& 1& a_{n+m}
\end{array}
\right), 1 \rangle=0$$

We also have

$(-1)^{i+1}A_{i,0}+ (-1)^{i+m} B_{i,1} + (-1)^{i+n}C_{i-1,1}=$
$$ \langle (\delta f)\left(
\begin{array}{cccccccccc} 
g(T^{i-1}_{i+m-1})& \prod\limits_{k=0}^{m-1}b_{i+k, i+m}& \ldots & \prod\limits_{k=0}^{m-1}b_{i+k, n+m-1} &1& \prod\limits_{k=0}^{m-1}b_{1,i+k} & \ldots &  \prod\limits_{k=0}^{m-1}b_{i-1,i+k} \\
1 & a_{i+m} & \ldots   & b_{i+m,n+m-1} &1 &b_{1,i+m} &\ldots&  \prod\limits_{k=0}^{m-1}b_{i-1,i+m}   \\
\ldots & \ldots & \ldots & \ldots&1& \ldots & \ldots& \ldots\\
\ldots & \ldots & \ldots & a_{n+m-1}&1&b_{1,n+m-1} &\ldots &b_{i-1,n+m-1}\\
\ldots & \ldots & \ldots & 1&a_{n+m}&1& \ldots& 1\\
\ldots & \ldots & \ldots & 1&1&a_1 & \ldots& b_{1,i-1}\\
\ldots & \ldots & \ldots & \ldots&\ldots &\ldots & \ldots& \ldots\\
\ldots & \ldots & \ldots & 1&1&1 & \ldots& a_{i-1}\\
\end{array}
\right), 1 \rangle=0$$

and 

$(-1)^{i+1}A_{i, n-i} +(-1)^{i+m}B_{i, n-i+1} + (-1)^{i+n}C_{i-1, n-i+1}=$
$$
 {\langle
(\delta f) \left(
\begin{array}{cccccccccc}
a_{i} &\ldots &\ldots & b_{i, n-1}& \prod\limits_{k=0}^{m-1} b_{i, n+k} & 1 &b_{1, i}&\ldots & b_{i-1, i}\\
1  & a_{i+1}& \ldots &b_{i+1, n-1}& \prod\limits_{k=0}^{m-1} b_{i+1, n+k} & 1 & b_{1, i+1}&\ldots & b_{i-1, i+1}\\
\vdots & \vdots & \vdots & \vdots& \vdots &  1 & \vdots &  \vdots\\
1& \ldots &\ldots & a_{n-1}& \prod\limits_{k=0}^{m-1} b_{n-1, n+k} &1 &\ldots& \ldots & b_{i-1, n-1}\\
1 & \ldots  &\ldots &&g(T^{n-1}_{n+m-1}) & 1&\prod\limits_{k=0}^{m-1} b_{1, n+k}& \ldots& \prod\limits_{k=0}^{m-1} b_{i-1, n+k}\\
1 & \ldots & \ldots & \ldots & \ldots & a_{n+m}& 1& \ldots &1\\
1 & 1 & \ldots & \ldots  &\ldots &\ldots & a_1 & \ldots & b_{1, i-1}\\
\vdots & \vdots & \vdots & \vdots& \vdots   &   \vdots& \vdots &  \vdots\\
1 & 1 & \ldots & \ldots  &\ldots & &\ldots & \ldots & a_{i-1}\\
\end{array}
 \right), 1\rangle}
$$

Thus, we obtain
\begin{equation}\label{ABC2}
0= \sum_{1\leq i \leq n, 1 \leq j\leq n-1, i+j\leq n+1}(-1)^{(j-1) (m-1)+i(n+m-1)}\left((-1)^{i+1}A_{i, j-1} + (-1)^{i+m}B_{i, j}+ (-1)^{i+n}C_{i-1, j}\right)
\end{equation}

Rearranging the terms in the above sum, and using equation \ref{ABC0}, we get,
$$\begin{array}{ll}
0= &\sum\limits_{1 \leq i \leq n, 1 \leq j\leq n, i+j\leq n}(-1)^{(j-1) (m-1)+i(n+m)+1}(A_{i, j} +
B_{i, j}+ C_{i, j})\\
&-\sum_{i=1}^{n} (-1)^{m-1 +i(n+m)}A_{i, 0} - \sum_{i=1}^n (-1)^{n(m+1)+i(n+1)}B_{i, n-i+1}\\
&-\sum_{j=1}^{n} (-1)^{(j-1)(m-1)}C_{0, j}
\end{array}
$$

$$
= \langle \delta(H)\left(
 \begin{array}{ccccccc} 
 a_1& b_{1,2}& b_{1, 3}&\ldots & b_{1, n+m-1}\\
1& a_2 &b_{2,3}&\ldots &b_{2, n+m-1}\\
\vdots & \vdots &\vdots & \vdots & \vdots \\
1 & 1 & 1&\ldots & b_{n+m-2, n+m-1}\\
1 & 1 & 1& \ldots & a_{n+m-1}
\end{array}
\right), a_{n+m}\rangle
$$ 
$$
-(-1)^{m(n+1)} \langle (\rho^2(f\otimes g))\left(
 \begin{array}{ccccccc} 
 a_1& b_{1,2}& b_{1, 3}&\ldots & b_{1, n+m-1}\\
1& a_2 &b_{2,3}&\ldots &b_{2, n+m-1}\\
\vdots & \vdots &\vdots & \vdots & \vdots \\
1 & 1 & 1&\ldots & b_{n+m-2, n+m-1}\\
1 & 1 & 1& \ldots & a_{n+m-1}
\end{array}
\right), a_{n+m}\rangle
$$
$$
-(-1)^{m(n+1)}\langle (\Delta(f)\cdot g)\left(
 \begin{array}{ccccccc} 
 a_1& b_{1,2}& b_{1, 3}&\ldots & b_{1, n+m-1}\\
1& a_2 &b_{2,3}&\ldots &b_{2, n+m-1}\\
\vdots & \vdots &\vdots & \vdots & \vdots \\
1 & 1 & 1&\ldots & b_{n+m-2, n+m-1}\\
1 & 1 & 1& \ldots & a_{n+m-1}
\end{array}
\right), a_{n+m}\rangle
$$
$$
- \langle (f \circ g)\left(
 \begin{array}{ccccccc} 
 a_1& b_{1,2}& b_{1, 3}&\ldots & b_{1, n+m-1}\\
1& a_2 &b_{2,3}&\ldots &b_{2, n+m-1}\\
\vdots & \vdots &\vdots & \vdots & \vdots \\
1 & 1 & 1&\ldots & b_{n+m-2, n+m-1}\\
1 & 1 & 1& \ldots & a_{n+m-1}
\end{array}
\right), a_{n+m}\rangle.
$$

\end{proof}

\begin{prop} \label{bvopx} The family $\Delta =\{\Delta^\bullet :C^\bullet(A,B,\varepsilon)\longrightarrow C^{\bullet-1}(A,B,\varepsilon)\}$ determines a BV-operator on the 
homotopy $G$-algebra $C^\bullet(A,B,\varepsilon)$. 

\end{prop} 

\begin{proof} We consider $f\in Z^n(A,B,\varepsilon)$ and $g\in Z^m(A,B,\varepsilon)$. By definition (see \cite[$\S$ 3]{SS}), we know that
\begin{equation}\label{t2.8t}
[f, g]=f\circ g -(-1)^{(n-1)(m-1)}g \circ f \textrm{ }\in C^{m+n-1}(A,B,\varepsilon)
\end{equation} Applying Lemma \ref{H}, we know that the cochains
\begin{equation*}
\begin{array}{c}
f\circ g -(-1)^{(n-1)m} \Delta(f)\cdot g + (-1)^{(n-1)m} \rho^2(f\otimes g)\\
g\circ f -(-1)^{(m-1)n} \Delta(g)\cdot f + (-1)^{(m-1)n} \rho^2(g\otimes f)\\
\end{array}
\end{equation*} are coboundaries. 
From \eqref{t2.8t}, it now follows that
\begin{equation}\label{t2.9t} 
[f,g]-(-1)^{(n-1)m} \Delta(f)\cdot g + (-1)^{(n-1)m} \rho^2(f\otimes g) +(-1)^{(m-1)}\Delta(g)\cdot f+(-1)^{m}\rho^2(g\otimes f)
\end{equation} is a coboundary. Applying Lemma \ref{delta}, it follows from \eqref{t2.9t} that
\begin{equation*}
[f,g]-(-1)^{(n-1)m} \Delta(f)\cdot g + (-1)^{m(n-1)} \rho^2(f\otimes g) +(-1)^{(m-1)}\Delta(g)\cdot f+(-1)^{m(n-1)}\rho^1(f\otimes g)
\end{equation*} is a coboundary. Since $\rho^1(f\otimes g)+\rho^2(f\otimes g)=\Delta(f\cdot g)$, we get
\begin{equation*}
[f,g]-(-1)^{(n-1)m} \Delta(f)\cdot g + (-1)^{m(n-1)} \Delta(f\cdot g) +(-1)^{(m-1)}\Delta(g)\cdot f
\end{equation*} is a coboundary. Using the fact that the dot product is graded commutative, we can put $\Delta(g)\cdot f=(-1)^{n(m-1)}f\cdot \Delta(g)$. The result is now clear. 
\end{proof}

\begin{theorem}\label{BVGsec}
For  secondary cohomology classes  $\bar f \in H^n(A, B, \varepsilon)$ and $\bar g\in H^m(A, B, \varepsilon)$, the Gerstenhaber bracket is determined by
$$[\bar f, \bar g]= (-1)^{(n-1)m}\overline{( \Delta(f)\cdot g + (-1)^n f \cdot  \Delta(g)-\Delta(f \cdot g))}\in H^{m+n-1}(A,B,\varepsilon)$$ Here $f$ and $g$ are any cocycles 
representing the classes $\bar f$ and $\bar g$ respectively.
\end{theorem}

\begin{proof}
This follows directly by applying Theorem \ref{hombv} and Proposition \ref{bvopx}. 
\end{proof}

\end{document}